\documentclass{amsart}
\usepackage{fullpage, graphics}

\usepackage{ifxetex,ifluatex}
\if\ifxetex T\else\ifluatex T\else F\fi\fi T%
  \usepackage{fontspec}
\else
  \usepackage[T1]{fontenc}
  \usepackage[utf8]{inputenc}
  \usepackage{lmodern}
\fi

\usepackage{hyperref,amsmath,amssymb,amsthm,colonequals}

\newcommand{\CC}{\mathbb{C}}
\newcommand{\FF}{\mathbb{F}}
\newcommand{\QQ}{\mathbb{Q}}
\newcommand{\ZZ}{\mathbb{Z}}

\numberwithin{equation}{section}
\newtheorem{theorem}[equation]{Theorem}
\newtheorem{lemma}[equation]{Lemma}

\DeclareMathOperator{\qual}{qual}

\newcommand{\SageMath}{\textsc{SageMath}}

\newcommand{\arXiv}[3]{\href{https://arxiv.org/abs/#1}{arXiv:#1v#2} (#3)}

\newcommand{\mr}[1]{(\href{https://mathscinet.ams.org/mathscinet-getitem?mr=#1}{MR#1})}

\newcommand{\doi}[2]{\href{https://doi.org/#1}{#2}}

\begin{document}

\title{Abelian varieties over $\FF_2$ of prescribed order}
\author{Kiran S. Kedlaya}
\address{Department of Mathematics, University of California San Diego, La Jolla, CA 92093, United States of America}
\email{kedlaya@ucsd.edu}
\urladdr{https://kskedlaya.org}
\date{August 5, 2022}
\thanks{Thanks to Francesc Fit\'e for discussions that led to this work as well as that of \cite{howe-kedlaya}, and to the anonymous referee for valuable feedback on the readability of the \SageMath{} code.
Financial support was provided by NSF (grants DMS-1802161, DMS-2053473) and UC San Diego (Warschawski Professorship).}

\begin{abstract}
We prove that for every positive integer $m$, there exist infinitely many simple abelian varieties over $\FF_2$ of order $m$.
The method is constructive, building on the work of Madan--Pal in the case $m=1$ to produce an explicit sequence of Weil polynomials giving rise to abelian varieties over $\FF_2$ of order $m$.
This sequence itself depends on the choice of a suitable generalized binary representation of $m$; by making careful choices of this representation,
we can ensure that the the resulting sequence of polynomials have 2-adic Newton polygons which guarantee the existence of suitable irreducible factors.
\end{abstract}

\maketitle

\section{Introduction}

How can a given positive integer $m$ occur as the order of the group of rational points of an abelian variety $A$ over $\FF_q$ (or for short, the \emph{order of $A$})? While this is nominally a question in arithmetic algebraic geometry, it immediately translates into a pure matter of algebraic number theory. To wit, 
Weil's theorems on the zeta function of $A$ (e.g., see \cite{Milne15}) imply that
$\#A(\FF_q) = P(1)$ where $P(x)$ is the characteristic polynomial of Frobenius on $A$. The polynomial $P(x)$ is monic of degree $2g$ where $g = \dim(A)$ and its complex roots can be labeled $\alpha_1,\dots,\alpha_{2g}$ so that
\[
|\alpha_i| = \sqrt{q}, \qquad \alpha_{g+i} = \overline{\alpha}_i \quad (i=1,\dots,g).
\]
Moreover, the Honda--Tate theorem asserts that any polynomial $P(x)$ satisfying these conditions (plus a mild additional hypothesis, which is automatic if $q = p$ is prime) occurs for some abelian variety.
Understanding our original question thus becomes a matter of studying the space of \emph{Weil polynomials}; some important foundational work on this issue was done by DiPippo--Howe \cite{dipippo-howe}.

While it may seem at this point that our original question is mostly resolved by prior work, 
it should be emphasized that we did \emph{not} ask about a specific value of $g$, and this has a profound effect on the nature of the question. For example,
we may read off from Weil's results the bounds
\[
(\sqrt{q}-1)^{2g} \leq \#A(\FF_q) \leq (\sqrt{q}+1)^{2g};
\]
if we distinguish these intervals based on $g$, then as $q$ increases they become more and more separated, to the extend that $\#A(\FF_q)$ eventually determines $g$ uniquely.
By contrast, if we fix $q$, then as $g$ increases the intervals eventually start to overlap,
so specifying $\#A(\FF_q)$ does not fix $g$ at all.

Keeping this in mind, let us now narrow our original question and ask: for a given prime power $q$, which integers occur \emph{at least once} as the order of an abelian variety $A$ over $\FF_q$?
Howe--Kedlaya \cite{howe-kedlaya} showed that \emph{every} positive integer occurs as the order of an abelian variety over $\FF_2$, which can further be taken to be ordinary.
Building on this, van Bommel--Costa--Li--Poonen--Smith \cite{vanbommel-etc} showed that for any fixed $q$, every \emph{sufficiently large} positive integer occurs as the order of an abelian variety over $\FF_q$.
This abelian variety can be further taken to be ordinary, geometrically simple, and/or principally polarizable, and for each combination of conditions one can in principle establish an effective ``sufficiently large''  cutoff;
for example, for $q > 4$, every integer $m \geq q^{3\sqrt{q} \log q}$ occurs as the order of an ordinary abelian variety over $\FF_q$ \cite[Theorem~1.13(b)]{vanbommel-etc}, and the lower bound on $m$ is best possible up to replacing $3$ with a smaller constant \cite[Remark~1.15]{vanbommel-etc}.

Another natural question to ask is, for a fixed $\FF_q$, \emph{how often} a given order can occur. For $q > 2$, a result of Kadets \cite{kadets} implies that for all but finitely many simple abelian varieties $A$ over $\FF_q$,
\[
\#A(\FF_q) \geq 1.359^{\dim(A)};
\]
in particular, there are only finitely many simple abelian varieties over $\FF_q$ of any given order (regardless of dimension). It is natural to try to count these, but we do not address this here.

Instead, we focus on the case $q=2$ and prove the following theorem.
\begin{theorem} \label{thm:AV by order}
For every positive integer $m$, there exist \emph{infinitely many} simple abelian varieties over $\FF_2$ (of various dimensions) of order $m$.
\end{theorem}

One key motivation for Theorem~\ref{thm:AV by order} is the fact that it holds for $m=1$ by an old theorem of Madan--Pal \cite{madan-pal}. That result gives a complete classification (up to isogeny) of simple abelian varieties over $\FF_2$ of order $1$ using work of Robinson on algebraic integers with all conjugates in a short real interval \cite{robinson, robinson2}. The case $m=2$ resolves a question of Kadets \cite[\S 1]{kadets}; the general question was raised in \cite[\S 1]{howe-kedlaya}.

The first step towards Theorem~\ref{thm:AV by order} is to produce some sequences of Weil polynomials giving rise to abelian varieties over $\FF_2$ of order $m$ (without apparent common simple factors).
This builds upon the work of Madan--Pal, and also uses some careful choices of generalized binary representations of $m$ as in \cite{howe-kedlaya} and \cite{vanbommel-etc}, including the \emph{nonadjacent binary representations} of Reitwiesner \cite{reitwiesner}.
One convenient feature of the construction is that each sequence we produce satisfies a second-order linear recurrence (closely linked to the recurrence relation satisfied by Chebyshev polynomials); this implies that every irreducible factor shared by more than one term of the sequence
corresponds to an abelian variety of order 1 (Lemma~\ref{lem:poly factor of right norm}). Consequently, if infinitely many terms of our sequence have irreducible factors of bounded codegree, then all but finitely many of the corresponding simple abelian varieties have order $m$ (Lemma~\ref{lem:poly factor of bounded codegree}).
This observation by itself is enough to establish Theorem~\ref{thm:AV by order} for $m$ prime (Lemma~\ref{lem:polynomials with roots in interval prime}), and thus to answer the question of Kadets.

To finish the proof of Theorem~\ref{thm:AV by order}, we establish this irreducibility using 2-adic calculations (mostly Newton polygons). 
For $m$ even, we get by with a slight variant of the nonadjacent binary representation (Lemma~\ref{lem:polynomials with roots in interval even case});
for $m$ odd, we need a representation of a more restricted form, which we construct using a short computer calculation
(Lemma~\ref{lem:polynomials with roots in interval odd case}) in \SageMath{} \cite{sage}. We include the relevant \SageMath{} code as an appendix; it is also available as a Jupyter notebook from the author's web site.

We conclude this introduction with some discussion of related questions that we do not treat.
\begin{itemize}
\item
It is not clear whether our approach can be upgraded to ensure that the simple abelian varieties we obtain are ordinary, geometrically simple, or principally polarizable; the ordinary condition in particular is incompatible with our use of 2-adic methods. For order greater than 1, it is possible that a suitable adaptation of \cite[Construction 9.1]{vanbommel-etc} can be used for this purpose. Such a construction may also shed some light on the \emph{number} of isogeny classes of simple abelian varieties over $\FF_2$ of fixed dimension and order, which is known for dimensions up to 6 by the exhaustive tables in LMFDB; see \cite{lmfdb-article} for more on this data and its tabulation.

\item
For simple abelian varieties of order 1, one cannot hope to enforce the ordinary, geometrically simple, and principally polarizable conditions simultaneously because the Madan--Pal classification demonstrates a ``rigidity'' of these abelian varieties. In fact, it can be shown that \emph{no} simple abelian variety of order 1 is both ordinary and geometrically simple; see \cite{kedlaya-dnelly}.

\item
Marseglia--Springer \cite{marseglia-springer} consider the question of finding abelian varieties realizing specific \emph{groups} of rational points;
this is a problem of a somewhat different nature because the group of rational points of an abelian variety over a finite field is \emph{not} an isogeny invariant (whereas its order is).
Using the result of \cite{howe-kedlaya} and \cite{vanbommel-etc}, Marseglia--Springer show that every finite abelian group occurs as the group of rational points of some ordinary abelian variety over $\FF_2$, $\FF_3$, and $\FF_5$ (and a slightly weaker analogue over $\FF_4$). 
Using Theorem~\ref{thm:AV by order}, Marseglia--Springer show that any fixed finite abelian group occurs as the group of rational points of infinitely many pairwise coprime abelian varieties.

\item
In contrast with Theorem~\ref{thm:AV by order}, for any positive integer $m$ there are are only finitely many isomorphism classes of curves whose Jacobians have order $m$; this is even true if we vary over all finite fields (modulo the trivial exception of curves of genus 0 in the case $m=1$).
It would be interesting to identify these curves for some small values of $m$. For example, it is known from work of Madan--Queen \cite{madan-queen}, Stirpe \cite{stirpe}, Mercuri--Stirpe \cite{mercuri-stirpe}, and Shen--Shi \cite{shen-shi} that there are eight isomorphism classes of curves of positive genus whose Jacobians have order 1: one curve of genus 1 over each of $\FF_3$ and $\FF_4$, plus six more curves of genera at most 4 over $\FF_2$.

\end{itemize}

\section{Setup}

We first introduce the setup used by Madan--Pal to study abelian varieties over $\FF_2$ with small order, building on work of Robinson \cite{robinson}.
Throughout this paper, we consider the interval
\[
[a,b] \colonequals [3 -2\sqrt{2}, 3+2\sqrt{2}].
\]

\begin{lemma} \label{lem:AV from poly}
Let $P(x) \in \ZZ[x]$ be an irreducible monic polynomial
with all roots in $[a,b]$
and set
\[
Q(x) \colonequals (-1)^{\deg P(x)}P(3-x), \qquad R(x) \colonequals x^{\deg P(x)} Q(x+2x^{-1}).
\]
Then $R(x)$ occurs as the characteristic polynomial of Frobenius of some simple abelian variety $A$ over $\FF_2$ with $\#A(\FF_2) = (-1)^{\deg P(x)} P(0)$.
\end{lemma}
\begin{proof}
Put $m = (-1)^{\deg P(x)} P(0)$.
The conditions on $P$ imply that $Q(x) \in \ZZ[x]$ is a monic irreducible polynomial with all roots in the interval $[-2\sqrt{2}, 2\sqrt{2}]$ with $Q(3) = m$,
and then that $R(x) \in \ZZ[x]$ is a monic irreducible polynomial with all roots on the circle $|x| = \sqrt{2}$.
By the Honda--Tate theorem \cite{tate}, \cite{waterhouse-milne}, $R(x)$ occurs as the characteristic polynomial of Frobenius of some simple abelian variety $A$ over $\FF_2$ (there being no Brauer obstruction because we are working over a prime field and we avoid the exceptional case $R(x) = x^2 - 2$);
for any such $A$, we have $\#A(\FF_2) = R(1) = Q(3) = m$.
\end{proof}

For $n$ a positive integer, let $T_n(x) \in \ZZ[x]$ be the $n$-th Chebyshev polynomial of the first kind for the ``arithmetic'' normalization (i.e., the Dickson polynomials of the first kind with parameter $1$):
\[
T_n(2 \cos \theta) = 2 \cos n\theta.
\]
For $n \geq 0$, define the polynomial $f_n(x)$ of degree $2n$ by the formula
\[
f_n(x) \colonequals x^n T_n(x + x^{-1} - 4).
\]
Since $x \mapsto x + x^{-1} - 4$ maps $[a,b]$ two-to-one onto $[-2, 2]$, $f_n(x)$ has all roots in $[a,b]$.
In the ring
\begin{equation} \label{eq:solution ring}
R \colonequals \frac{\ZZ[x^{\pm 1}, y^{\pm 1},(x-1)^{-1}]}{(x + x^{-1} - 4 - y - y^{-1})}
\cong \ZZ[x^{\pm 1}, (x-1)^{-1}, \sqrt{x^2-6x+1}],
\end{equation}
we have
\begin{equation} \label{eq:fn equation}
f_n(x) = x^n T_n(y + y^{-1}) = x^n(y^n + y^{-n}).
\end{equation}

We finally introduce a key modification that will give rise to abelian varieties of prescribed orders greater than 1.
For $n,k \geq 0$, define the rational function
\[
g_{n,k}(x) \colonequals (x-1)^{-k} \sum_{j=0}^k \binom{k}{j} f_{n+j}(x),
\]
so that  $g_{n,0}(x) = f_n(x)$.  In the ring $R$, we have
\begin{equation}
g_{n,k}(x) = (xy)^{n} \left( \frac{xy+1}{x-1} \right)^k + (xy^{-1})^n \left( \frac{xy^{-1}+1}{x-1} \right)^k.
\label{eq:gnk algebraic representation}
\end{equation}
We will see later that $g_{n,k}(x)$ is a polynomial of degree $2n+k$ (Lemma~\ref{lem:gnk polynomial}) with constant term $(-2)^k$
\eqref{eq:gn mod x2 congruence} having all roots in $[a,b]$ (Lemma~\ref{lem:hn roots interval general}).

\section{Recurrence relations and algebraic corollaries}

We next introduce some recurrence relations satisfied by $f_n(x)$ and $g_{n,k}(x)$, and use these to derive some additional algebraic properties, notably that  $g_{n,k}(x)$ is indeed a polynomial (Lemma~\ref{lem:gnk polynomial}). Many spot verifications of these properties can also be found in the associated Jupyter notebook.

To begin with, recall that the Chebyshev polynomials are characterized by the recurrence relation and initial conditions:
\begin{equation} \label{eq:Tn recurrence relation}
T_n(x) - xT_{n-1}(x) + T_{n-2}(x) = 0, \quad T_0(x) = 2, \quad T_1(x) = x.
\end{equation}
This translates into a corresponding recurrence relation and initial conditions for $f_n(x)$:
\begin{equation} \label{eq:fn recurrence relation}
f_n(x) - (x^2 - 4x + 1)f_{n-1}(x) + x^2 f_{n-2}(x) = 0, \quad f_0(x) = 2, \quad f_1(x) = x^2 - 4x + 1.
\end{equation}
From \eqref{eq:fn recurrence relation}, it is easy to deduce by induction that 
\begin{align} 
\label{eq:fn mod x2 congruence}
f_n(x) &\equiv x^{2n} + 1 \pmod{x} \\
\label{eq:fn value at 1}
f_n(x) &\equiv (-1)^n 2 \pmod{x-1} \\
\label{eq:fn mod 8 congruence}
f_n(x) &\equiv x^{2n} + 4n(x^{2n-1} + x^{2n-3} + \cdots + x) + 1 \pmod{8}.
\end{align}
(In \eqref{eq:fn mod x2 congruence}, the term $x^{2n}$ is only relevant when $n=0$.)
The recurrence relation \eqref{eq:fn recurrence relation} for $f_n$ translates into the recurrence relation
\begin{equation} \label{eq:gn recurrence relation}
g_{n,k}(x) - (x^2 - 4x + 1)g_{n-1,k}(x) + x^2 g_{n-2,k}(x) = 0.
\end{equation}
We can also formulate recurrence relations for $g_{n,k}$ in which $k$ varies. To begin with, for $k \geq 1$,
\begin{equation} \label{eq:gn simple recursion in k}
(x-1) g_{n,k}(x) = g_{n,k-1}(x) + g_{n+1,k-1}(x).
\end{equation}
We can also avoid division by $x-1$ at the expense of lengthening the recurrence in the $k$-aspect.
\begin{lemma}
For $k \geq 2$,
\begin{gather}
\label{eq:gn recurrence relation by k}
g_{n,k}(x) - (x-3)g_{n,k-1}(x) + 2 g_{n,k-2}(x) = 0, \\
\label{eq:gn recurrence relation by k2}
g_{n,k}(x) + 4 g_{n,k-1}(x) + 4g_{n,k-2}(x) - x^2 g_{n-1,k}(x) = 0.
\end{gather}
\end{lemma}
\begin{proof} 
The equalities can be seen to hold for $k=2$ by expanding $g_{n,k}(x)$ in terms of $f_{n+j}(x)$ and applying \eqref{eq:fn recurrence relation},
and then for $k > 2$ by induction using \eqref{eq:gn simple recursion in k}.
\end{proof}

We are now ready to establish that $g_{n,k}(x)$ is in fact a polynomial.
\begin{lemma} \label{lem:gnk polynomial}
For $n,k \geq 0$, $g_{n,k}(x)$ is a polynomial of degree $2n+k$.
\end{lemma}
\begin{proof}
We have $g_{n,k}(x) \in \ZZ[x]$ for $k=0$ because $g_{n,0}(x) = f_n(x)$, and for $k=1$ by \eqref{eq:fn value at 1} and \eqref{eq:gn simple recursion in k}.
By \eqref{eq:gn recurrence relation by k} we deduce that $g_{n,k}(x) \in \ZZ[x]$ for $n,k \geq 0$.
The degree assertion then follows from the fact that $\deg f_{n+j}(x) = 2n+2j$.
\end{proof}

Using \eqref{eq:gn simple recursion in k}, we may formally promote \eqref{eq:fn mod x2 congruence}: for $n > 0$,
\begin{equation} \label{eq:gn mod x2 congruence}
g_{n,k}(x) \equiv (-2)^k \pmod{x}.
\end{equation}
We may also promote \eqref{eq:fn value at 1} as follows.
\begin{lemma} \label{lem:2-adic congruence at shift}
For $n, k \geq 0$,
\begin{equation} \label{eq:value at 1}
g_{n,k}(x) \equiv (-1)^{n-k} ((1+i)^k + (1-i)^k) \pmod{x-1}.
\end{equation}
\end{lemma}
\begin{proof}
For $k=0$ this is a restatement of \eqref{eq:fn value at 1}.
For $k=1$, we may check the claim for $n=0,1$ from the values
\[
f_0(x) = 2, \quad f_1(x) = x^2 - 4x + 1, \quad f_2(x) = x^4 - 8x^3 + 16x^2 - 8x + 1
\]
and then for general $n$ by \eqref{eq:gn recurrence relation}.
We may then extend to general $k$ using \eqref{eq:gn recurrence relation by k}.
\end{proof}

We next consider analogues of \eqref{eq:fn mod 8 congruence} for $g_{n,k}$ for $k>0$. We start with a mod 2 congruence: from \eqref{eq:fn mod 8 congruence}
and \eqref{eq:gn simple recursion in k},
\begin{equation}
\label{eq:gnk mod 2 leading term}
g_{n,k}(x) \equiv x^{2n} (x+1)^k + 2^k \pmod{2}.
\end{equation}
We can also establish congruences modulo a higher power of 2 provided that we ignore some leading coefficients.

\begin{lemma}
For $n \geq 0$,
\begin{equation} \label{eq:gnk mod 2 congruence}
g_{n,k}(x) \equiv 0 \pmod{(x^{2n}, 2^{k})}.
\end{equation}
\end{lemma}
\begin{proof}
The claim holds for $n=0$ and $k=0$ vacuously, and for $k=1$ by \eqref{eq:gnk mod 2 leading term}.
We may then deduce the general case by \eqref{eq:gn recurrence relation by k2}.
\end{proof}

Finally, from \eqref{eq:fn mod 8 congruence} and \eqref{eq:gn simple recursion in k} we obtain some congruences modulo higher powers of 2 relative to $k$:
\begin{align} \label{eq:gn1 mod 4 congruence}
g_{n,1}(x) &\equiv \sum_{i=0}^{2n-1} (-1)^{\lfloor (i-1)/2 \rfloor}  2x^i \pmod{(x^{2n}, 8)} \\
 \label{eq:gnk mod 8 congruence}
g_{n,2}(x) &\equiv \sum_{i=0}^{n-1} 4x^{2i} \pmod{(x^{2n}, 8)}.
\end{align}

\section{Counting roots}

To count zeros of polynomials in the interval $[a,b]$, we use an approach based on \emph{winding numbers}.

\begin{lemma}\label{lem:hn roots interval general}
Let $a_0, \dots, a_k$ be a sequence of real numbers with $a_k = 1$, such that the polynomial $Q(z) \colonequals \sum_{i=0}^k a_i z^i$ has all of its complex roots inside the closed disc $|z| \leq \sqrt{2}$
(e.g., by condition \eqref{eq:weight condition} below).
Then for each $n \geq 0$, the roots of the polynomial
\[
P_n(x) \colonequals \sum_{i=0}^k a_i g_{n,i}(x)
\]
are all real and contained in $[a,b]$. If in fact $Q(z)$ has all of its complex roots inside the open disc $|z| < \sqrt{2}$,
then the roots of $P_n(x)$ are pairwise distinct.
\end{lemma}
\begin{proof}
By continuity (of the roots of a polynomial as a function of the coefficients), we may reduce to the case where $Q(z)$ has all of its complex roots in the open disc $|z| < \sqrt{2}$.
For $\theta \in [-2\pi, 2\pi]$, we define a parametric complex solution of the equation
\[
x + x^{-1} - 4 = y + y^{-1}
\]
by setting $y(\theta) \colonequals e^{2 \pi i \theta}$ and 
\[
x(\theta) \colonequals \cos \theta + 2 + \sqrt{\cos^2 \theta + 4 \cos \theta + 3},
\]
choosing the branch of the square root so that $x(\theta)$ varies continuously and
\[
x(-2\pi) = b, x(0) = a, x(2\pi) = b.
\]
Define the function
\[
s(\theta) \colonequals \begin{cases} \frac{x(\theta) y(\theta) + 1}{x(\theta)-1} & (\theta \neq \pm \pi) \\
-1-i & (\theta = -\pi) \\
-1+i & (\theta = \pi);
\end{cases}
\]
one may check using L'H\^opital's rule that this function is continuous. By writing
\begin{equation} \label{eq:norm of moebius}
\left| s(\theta) \right|^2 = \frac{(x(\theta)y(\theta)+1)(x(\theta)y(-\theta)+1)}{(x(\theta)- 1)^2}
=  \frac{x(\theta)^2 + x(\theta)(x(\theta) + x(\theta)^{-1} - 4) + 1}{x(\theta)^2 - 2x(\theta) + 1} = 2,
\end{equation}
we deduce that $s$ carries $[-2\pi, 2\pi]$ into the circle $|z| = \sqrt{2}$ (making one full counterclockwise circuit).

By \eqref{eq:gnk algebraic representation},
\begin{equation} \label{eq:hn real expression}
P_n(x(\theta)) = 2 x(\theta)^{n} \mathrm{Real} \left( y(\theta)^{n} \sum_{i=0}^k a_i s(\theta)^{i} \right).
\end{equation}
Since $x(\theta)$ is monotone, the zeros of $P_n(x)$ in the interval $[a,b]$ (counted without multiplicity) are in bijection with zeros of 
$P_n(x(\theta))$ in either of the intervals $[-2\pi, 0]$ or $[0, 2\pi]$.
We will estimate the number of zeros of $P_n(x(\theta))$
by computing the displacement of 
\[
\arg \left(y(\theta)^{n} \sum_{i=0}^k a_i s(\theta)^{i}
\right) = n \arg y(\theta) + k \arg s(\theta) + \arg
\left( \sum_{i=0}^k a_i s(\theta)^{i-k} \right)
\]
over the interval $[-2\pi, 2\pi]$ (choosing all of the arguments to vary continuously in $\theta$).

As $\theta$ varies from $-2\pi$ to $2\pi$, the displacement of $n \arg y(\theta) + k \arg s(\theta)$ equals $(4n+2k) \pi$. Meanwhile,
we may see that $\arg \left( \sum_{i=0}^k a_i s(\theta)^{i-k} \right)$ has displacement 0
by combining \eqref{eq:norm of moebius}, our condition on the roots of $Q(z)$, and the argument principle.

Since $\arg \left(y(\theta)^{n} \sum_{i=0}^k a_i s(\theta)^{i} \right)$ varies continuously from 0 to $(4n+2k) \pi$ as $\theta$ runs from $2\pi$ to $2\pi$, by the intermediate value theorem it evaluates to an odd multiple of $\pi$ at no fewer than $4n+2k$ distinct values in this range.
By \eqref{eq:hn real expression}, these values are zeros of
$P_n(x(\theta))$ in $[a,b]$, each counted at most twice. Since $P_n(x)$ is a polynomial of degree $2n+k$, we deduce that all of its zeros are pairwise distinct real numbers in $[a,b]$.
\end{proof}

Note that in Lemma~\ref{lem:hn roots interval general}, one way to enforce the condition on $Q(z)$ is to assume
\begin{equation} \label{eq:weight condition}
\sum_{i=0}^{k-1} |a_i| 2^{(i-k)/2} \leq 1,
\end{equation}
as then the triangle inequality implies that $|z^{-k} Q(z)| > 0$ for $|z| > \sqrt{2}$
(compare \cite[Lemma~2]{howe-kedlaya}).
This restricted setting will be enough to prove Theorem~\ref{thm:AV by order} for $m$ even (Lemma~\ref{lem:polynomials with roots in interval even case}),
but we will need to exercise more flexibility for $m$ odd (Lemma~\ref{lem:polynomials with roots in interval odd case}).

\section{Repeated zeros in a recurrent sequence}

Note that for any fixed sequence $\{a_i\}$, the sequence of polynomials $P_n(x)$ considered in Lemma~\ref{lem:hn roots interval general} satisfies the same second-order recurrence as the ones satisfied by $f_n(x)$ \eqref{eq:fn recurrence relation} and $g_{n,k}(x)$ \eqref{eq:gn recurrence relation}.
Using this, we can show that the polynomials $P_n(x)$ have very few common zeros.

\begin{lemma} \label{lem:repeated roots}
Let $\{P_n(x)\}_{n \geq 0}$
be a sequence of monic integer polynomials satisfying the recurrence relation
\begin{equation} \label{eq:P recurrence relation}
P_n(x) - (x^2 - 4x + 1)P_{n-1}(x) + x^2 P_{n-2}(x) = 0.
\end{equation}
Suppose that $\alpha \in \CC^\times$ is a root of both $P_n(x)$ and $P_{n'}(x)$ for some $n < n'$.
Then $\alpha$ is a unit in the ring of algebraic integers.
\end{lemma}
\begin{proof}
In the ring $R$ from \eqref{eq:solution ring}, we can solve the recurrence 
\eqref{eq:P recurrence relation} to obtain an analogue of \eqref{eq:fn equation}: for some $P_+, P_- \in R$ (independent of $n$),
\[
P_n = P_+(xy)^n + P_- (xy^{-1})^n.
\]
Define a specialization homomorphism $\pi: R \to \CC$ taking $x$ to $\alpha$ by picking a square root of $\alpha + \alpha^{-1} - 4$;
then solving the system of equations
\[
\pi(P_+) \pi(xy)^n + \pi(P_-) \pi(xy^{-1})^n = \pi(P_+) \pi(xy)^{n'} + \pi(P_-) \pi(xy^{-1})^{n'} = 0
\]
yields
\[
\pi(xy)^{n'-n} = \pi(xy^{-1})^{n'-n}
\]
and so $\pi(y)^{2(n'-n)} = 1$. By \eqref{eq:fn equation} this yields $f_{2(n'-n)}(\alpha) = 2 \alpha^{2(n'-n)}$; since 
$\deg f_{2(n'-n)}(x) = 4(n'-n) > 2(n'-n)$ and $f_{2(n'-n)}(0) = 1$ by \eqref{eq:fn mod x2 congruence}, 
$\alpha$ is a root of the monic polynomial $f_{2(n'-n)}(x) - 2x^{2(n'-n)}$ with constant coefficient 1.
\end{proof}

This has the following implication for Theorem~\ref{thm:AV by order}.
\begin{lemma} \label{lem:poly factor of right norm}
Let $m>1$ be an integer and fix a sequence $a_0,\dots,a_k$
of integers satisfying the hypotheses of
Lemma~\ref{lem:hn roots interval general}.
Suppose that for infinitely many $n$, the polynomial
$P_n(x) = \sum_i a_i g_{n,i}(x)$ over $\QQ$
has an irreducible factor $Q(x)$ with $P(0) = \pm m$. Then
there exist infinitely many simple abelian varieties $A$ over $\FF_2$ with $\#A(\FF_2) = m$.
\end{lemma}
\begin{proof}
By Lemma~\ref{lem:hn roots interval general}, the polynomial $P_n(x)$ has all of its roots in $[a,b]$, as then does $Q(x)$. 
By Lemma~\ref{lem:repeated roots}, the factors $Q(x)$ are pairwise distinct. We may thus apply Lemma~\ref{lem:AV from poly} to conclude.
\end{proof}

We also need a slightly modified version of Lemma~\ref{lem:poly factor of right norm}.
\begin{lemma} \label{lem:poly factor of bounded codegree}
Let $m>1$ be an integer and fix a sequence $a_0,\dots,a_k$ of integers satisfying the hypotheses of 
Lemma~\ref{lem:hn roots interval general},
and additionally satisfying $\sum_{i=0}^k a_i 2^i \in \{ \pm m\}$.
Suppose that for infinitely many $n$, the polynomial
$P_n(x) = \sum_i a_i g_{n,i}(x)$ over $\QQ$
has a monic irreducible factor $Q(x)$ whose codegree
(i.e., $\deg P_n(x) - \deg Q(x)$) is bounded by a function of $m$ alone.
Then
there exist infinitely many simple abelian varieties $A$ over $\FF_2$ with $\#A(\FF_2) = m$.
\end{lemma}
\begin{proof}
By hypothesis, we can write
$P_n(x) = Q(x) R(x)$ where $\deg R(x)$
is bounded by a function of $m$ alone.
By Lemma~\ref{lem:hn roots interval general}, $P_n(x)$ has all roots in $[a,b]$, as then do $Q(x)$ and $R(x)$.
Since $R(x)$ has integer coefficients and roots in a fixed interval,
$R(x)$ itself is contained in a finite set determined by $m$.
By Lemma~\ref{lem:repeated roots}, there are only finitely many values of $n$ for which $R(x)$ has constant term not in $\{\pm 1\}$.
For the remaining values, $Q(x)$ 
is a monic irreducible polynomial with $Q(0) = \pm m$. We may thus apply Lemma~\ref{lem:poly factor of right norm} to conclude.
\end{proof}

\section{Nonadjacent binary representations}

In order to apply Lemma~\ref{lem:hn roots interval general}, we need to find ways to represent a given positive integer $m$ as the evaluation at $z=2$
of a monic integer polynomial $Q(z)$ having all complex roots in the disc $|z| \leq \sqrt{2}$. That is, we need a \emph{binary representation} of $m$
which is ``efficient'' in a suitable sense.

One good candidate is the \emph{nonadjacent binary representation} of $m$ in the sense of Reitwiesner \cite{reitwiesner}:
\begin{equation} \label{eq:nonadjacent rep}
m = \sum_{i=0}^\infty a_i 2^i \qquad \mbox{where} \qquad a_i \in \{-1,0,1\}, \, a_k = 1, \, a_i a_{i+1} = 0 \quad (i \geq 0).
\end{equation}
The sequence $a_0,\dots,a_k$ can be generated efficiently from $m$ using the rule
\[
a_0 = \begin{cases} \pm 1 & m \equiv \pm 1 \pmod{4} \\
0 & m \equiv 0 \pmod{2}. 
\end{cases}
\]
Moreover, the largest index $k$ with $a_k \neq 0$ (and hence $a_k=1$) is $k(m) = \lfloor \log_2 (3m) \rfloor - 1$. 

Define the polynomial
\[
h_{n,m}(x) \colonequals \sum_{i=0}^k (-1)^{i+k} a_i g_{n,i}(x);
\]
by Lemma~\ref{lem:gnk polynomial}, $h_{n,m}(x)$ is a monic polynomial of degree $2n+k$.
By \eqref{eq:gn mod x2 congruence},
\begin{equation} \label{eq:hnm value}
h_{n,m}(0) = (-1)^k m.
\end{equation}
Since we chose $a_0,\dots,a_k$ without reference to $n$,
we deduce from \eqref{eq:gn recurrence relation} that
\begin{equation} \label{eq:hn recurrence relation}
h_{n,k}(x) - (x^2 - 4x + 1)h_{n-1,k}(x) + x^2 h_{n-2,k}(x) = 0.
\end{equation}

\begin{lemma} \label{lem:hn roots interval}
The roots of the polynomial $h_{n,m}(x)$ are all real, pairwise distinct, and contained in the interval $[a,b] = [3-2\sqrt{2}, 3+2\sqrt{2}]$.
\end{lemma}
\begin{proof}
From the definition of nonadjacent binary representations, we see that
\begin{equation} \label{eq:bound from nonadjacent}
\sum_{i=0}^{k-1} |a_i| 2^{(i-k)/2} \leq 1 + 2^{-1} + \cdots + 2^{-\lfloor k/2 \rfloor} \leq 1 - 2^{-k/2} < 1.
\end{equation}
We may thus apply Lemma~\ref{lem:hn roots interval general}.
\end{proof}

In passing, we can already derive some cases of Theorem~\ref{thm:AV by order}, including the case $m=2$ considered in \cite{kadets}.
While this case is logically necessary for the rest of the proof, it does illustrate the key ideas with limited technical complications compared to the general case.

\begin{lemma} \label{lem:polynomials with roots in interval prime}
Theorem~\ref{thm:AV by order} holds when $m$ is prime.
\end{lemma}
\begin{proof}
This is immediate from Lemma~\ref{lem:poly factor of right norm}
and Lemma~\ref{lem:hn roots interval}:
if $m$ is prime, then $h_{n,m}(x)$ admits a unique irreducible factor with constant coefficient $\pm m$.
\end{proof}

\section{2-adic congruences: even order case}

In this section, we prove Theorem~\ref{thm:AV by order} for $m$ even, using factorizations over the 2-adic field $\QQ_2$.
Let $v_2(m)$ denote the $2$-adic valuation of $m$.
By convention, our Newton polygons are convex with left endpoint $(0,0)$.

As a warmup, we treat the case where $v_2(m) = 1$.
\begin{lemma} \label{lem:polynomials with roots in interval even}
Theorem~\ref{thm:AV by order} holds when $m \equiv 2 \pmod{4}$.
\end{lemma}
\begin{proof}
From \eqref{eq:gnk mod 2 leading term},
\begin{equation}
\label{eq:hn mod 2 leading term}
h_{n,m}(x) \equiv x^{2n} \sum_{i=0}^{k(m)} a_i (x+1)^i + m\pmod{ 2}.
\end{equation}
By \eqref{eq:hnm value} and \eqref{eq:hn mod 2 leading term}, the 2-adic Newton polygon of $h_{n,m}(x)$ has vertices
\[
(0,0), (k(m)-d, 0), (2n+k(m), 1)
\]
for some $d \in \{0,\dots,k(m)\}$. 
The last segment corresponds to an irreducible factor of $h_{n,m}(x)$ over $\QQ_2$;
hence over $\QQ$, $h_{n,m}(x)$ has an irreducible factor of codegree bounded by a function of $m$.
We may thus combine Lemma~\ref{lem:poly factor of bounded codegree} and Lemma~\ref{lem:hn roots interval} to conclude.
\end{proof}

We next generalize the Newton polygon calculation from the previous argument.
\begin{lemma} \label{lem:2-adic NP}
For $m$ even and $n \gg 0$, the $2$-adic Newton polygon of $h_{n,m}(x)$ has vertices
\[
(0,0), (k(m)-d, 0), (2n+ k(m), v_2(m))
\]
where $d$ is the order of vanishing of $\sum_{i=0}^k a_i (x+1)^k$ at $x=0$ over $\FF_2$.
\end{lemma}
\begin{proof}
From \eqref{eq:gnk mod 2 congruence}, we have
\begin{equation} \label{eq:hn mod 2 congruence}
h_{n,m}(x) \equiv 0 \pmod{(x^{2n}, 2^{v_2(m)})}.
\end{equation}
By combining \eqref{eq:hn mod 2 leading term} with \eqref{eq:hn mod 2 congruence}, we deduce the claim.
\end{proof}

This gives us a direct adaptation of  Lemma~\ref{lem:polynomials with roots in interval even} when $v_2(m)$ is odd.
\begin{lemma}  \label{lem:polynomials with roots in interval odd 2-adic valuation}
Theorem~\ref{thm:AV by order} holds when $v_2(m)$ is odd.
\end{lemma}
\begin{proof}
Define $d$ as in Lemma~\ref{lem:2-adic NP}.
Since $v_2(m)$ is odd, by restricting $n$ to a suitable arithmetic progression we can ensure that
$\gcd(v_2(m), 2n+d) = 1$; then the final segment of the 2-adic Newton polygon of $h_{n,m}(x)$ corresponds to an irreducible factor of $h_{n,m}(x)$ over $\QQ_2$.
For such $n$, $h_{n,m}(x)$ has an irreducible factor over $\QQ$ of codegree bounded by a function of $m$;
we may thus combine Lemma~\ref{lem:poly factor of bounded codegree} and Lemma~\ref{lem:hn roots interval} to conclude.
\end{proof}

To handle the case where $v_2(m)$ is even, it is convenient to separate off the case $v_2(m) = 2$, which we can handle in a similar manner.

\begin{lemma}  \label{lem:polynomials with roots in interval valutation 2}
Theorem~\ref{thm:AV by order} holds when $v_2(m) = 2$.
\end{lemma}
\begin{proof}
Define $d$ as in Lemma~\ref{lem:2-adic NP}.
The final segment of the 2-adic Newton polygon of $h_{n,m}(x)$ corresponds to either an irreducible factor of $h_{n,m}(x)$ over $\QQ_2$
or a pair of irreducible factors, each of degree $1/(n+d/2)$. 
In the latter case (which only occurs if $d$ is even), the coefficient of $x^{n+d/2}$ must be congruent to 0 or 4 modulo 8 according to whether $(-1)^k m$ is congruent to $-4$ or $4$ modulo $16$;
since \eqref{eq:gnk mod 2 congruence}  and \eqref{eq:gnk mod 8 congruence} together imply
\begin{equation} \label{eq:hn mod 8 congruence}
h_{n,m}(x) \equiv 4 (x^{2n} + \cdots + x^2+1) \pmod{(x^{2n}, 8)},
\end{equation}
this case can be ruled out by fixing the parity of $n$ appropriately. For such $n$, $h_{n,m}(x)$ has an irreducible factor over $\QQ$ of codegree bounded by a function of $m$.
We may thus combine Lemma~\ref{lem:poly factor of bounded codegree} and Lemma~\ref{lem:hn roots interval} to conclude.
\end{proof}

To handle higher values of $v_2(m)$, we modify the polynomial $h_{n,m}(x)$ so that we can better emulate the case $v_2(m) = 1$.
\begin{lemma}  \label{lem:polynomials with roots in interval even 2-adic valuation}
Theorem~\ref{thm:AV by order} holds when $v_2(m) \geq 4$.
\end{lemma}
\begin{proof}
Since $a_0 = a_1 = a_2 = a_3 = 0$, the sequence
\[
(a'_0, \dots, a'_k) = (2,1,0,0,a_4,\dots,a_k)
\]
satisfies 
\[
\sum_{i=0}^k |a'_i| 2^{(i-k)/2} \leq 2\cdot 2^{-k/2} + 2^{(1-k)/2} + 1 - 2^{(4-k)/2} < 1;
\]
hence the polynomial
\[
h'_{n,m}(x) = h_{n,m}(x) + (-1)^k (2g_{n,0}(x) + g_{n,1}(x))
\]
satisfies the hypothesis of Lemma~\ref{lem:hn roots interval general}. We may again compute its 2-adic Newton polygon using \eqref{eq:gnk mod 2 leading term}, \eqref{eq:gn1 mod 4 congruence}, \eqref{eq:hn mod 2 leading term}, and \eqref{eq:hn mod 2 congruence}:
its vertices are
\[
(0,0), (k-d, 0), (2n+k(m)-1, 1), (2n+k(m), v_2(m))
\]
where $d$ is the order of vanishing of $x+1+\sum_{i=0}^k a_i (x+1)^k$ at $x=0$ over $\FF_2$.
Over $\QQ_2$, the middle segment corresponds to a single irreducible factor of $h_{n,m}(x)$;
we may thus argue as in the proof of Lemma~\ref{lem:polynomials with roots in interval odd 2-adic valuation} to conclude.
\end{proof}

To summarize, by combining Lemma~\ref{lem:polynomials with roots in interval odd 2-adic valuation},
Lemma~\ref{lem:polynomials with roots in interval valutation 2}, 
and Lemma~\ref{lem:polynomials with roots in interval even 2-adic valuation},
we deduce the following.
\begin{lemma} \label{lem:polynomials with roots in interval even case}
Theorem~\ref{thm:AV by order} holds when $m$ is even.
\end{lemma}

\section{2-adic congruences: odd order case}

In this section, we prove Theorem~\ref{thm:AV by order} for $m$ odd.
For this, we cannot use the 2-adic Newton polygon of $h_{n,m}(x)$ because it has all slopes equal to 0; instead, we use the 2-adic Newton polygon of $h_{n,m}(x+1)$.
To begin with, note that for $m$ odd, by \eqref{eq:value at 1} we have
\begin{equation} \label{eq:hnm at 1 when m odd}
h_{n,m}(1) \equiv 2 \pmod{4};
\end{equation}
more precisely, we are using here the fact that $a_0$ is odd, $a_1$ is even, and $g_{n,k}(1) \equiv 0 \pmod{4}$ for $k \geq 2$.

To illustrate the method, we first prove some isolated cases of Theorem~\ref{thm:AV by order}.
\begin{lemma} \label{lem:additional1}
Theorem~\ref{thm:AV by order} holds for any odd $m$ such that $k(m)$ is even and
\[
\sum_{i=0}^{k(m)} a_i x^i \equiv (x+1)^{k(m)} \pmod{2}.
\]
For example, this holds for $m = 15, 45, 51, 75, 77, 85$.
\end{lemma}
\begin{proof}
For $n = 2^j-k(m)/2$, we have from \eqref{eq:hn mod 2 leading term} that
\[
h_{n,m}(x) \equiv (x+1)^{2^{j+1}} \pmod{2}.
\]
By Lemma~\ref{lem:2-adic congruence at shift} and \eqref{eq:hnm at 1 when m odd}, the Newton polygon of $h_{n,m}(x+1)$  has vertices
\[
(0,0), (2^{j+1}, 1);
\]
that is, $h_{n,m}(x+1)$ satisfies the Sch\"onemann--Eisenstein irreducibility criterion at 2.
We may thus combine Lemma~\ref{lem:AV from poly} and Lemma~\ref{lem:hn roots interval} to conclude.
\end{proof}

To cover the remaining values of $m$, we use a variant construction that preserves the key features of this method.
We say that a monic integer polynomial $Q(z)$ is a \emph{compliant representation} of the odd positive integer $m$ if
$Q(2) = m$, $Q(z) \equiv (z-1)^{\deg Q(z)} \pmod{2}$, and $Q(z)$ has all complex roots in the disc $|z| < \sqrt{2}$.

\begin{lemma} \label{lemma:mod 2 criterion}
Theorem~\ref{thm:AV by order} holds for $m$ admitting a compliant representation.
\end{lemma}
\begin{proof}
Let $Q(z)$ be a compliant representation of $m$; by multiplying by $z-1$ as needed, we may ensure that $k \colonequals \deg Q(z)$ is even.
Write $Q(z) = \sum_{i=0}^k c_i z^i$.
For each $n$, the polynomial $P_n(x) = \sum_i (-1)^{i+k} c_i g_{n,i}(x)$ satisfies $P_n(0) = (-1)^{k} m$.
By Lemma~\ref{lem:hn roots interval general}, $P_n(x)$ has all roots in $[a,b]$.
By \eqref{eq:value at 1} (as in the proof of  \eqref{eq:hnm at 1 when m odd}), $P_n(1) \equiv 2 \pmod{4}$.
For $n = 2^j-k/2$, we see from the proof of Lemma~\ref{lem:additional1} that $P_n(x)$ is Eisenstein at 2 and hence irreducible.
We may thus directly apply Lemma~\ref{lem:AV from poly} to conclude.
\end{proof}

In order to produce compliant representations, it will be convenient to further quantify the condition on the roots. To this end, for $Q(z)$ a compliant representation of some integer $m$, define the \emph{quality} of $Q(z)$ as
\begin{equation} \label{eq:quality}
\qual(Q(z)) \colonequals \min\{|Q(z)|: |z| = \sqrt{2}\} =  \min\{|Q(z)|: |z| \geq \sqrt{2}\}.
\end{equation}
Keep in mind that the last equality in \eqref{eq:quality} is a consequence of the maximum modulus principle, and is only valid under the assumption that $Q(z)$ is compliant.

\begin{lemma} \label{lem:compliant representation}
Let $m$ be a positive odd integer.
\begin{enumerate}
\item[(a)] If $m \leq 3094$, then $m$ admits a compliant representation.
\item[(b)] If $3094 \leq m \leq 50000$, then $m$ admits a compliant representation of quality at least $7$.
\end{enumerate}
\end{lemma}
\begin{proof}
We describe a computer-assisted proof; the associated computations run in \SageMath{} (version 9.6) in under 5 minutes on a standard laptop (we used one core on an Intel iCore i5-6200U @2.30GHz). As the \SageMath{} code is quite short, we have included it in its entirety in the appendix.

We first observe that \SageMath{} provides an exact representation of the subfield $\overline{\QQ}$ of $\CC$ based on interval arithmetic. Using this, given a monic integer polynomial $Q(z)$,
we may compute the roots of $Q(z)$ in $\overline{\QQ}$ and then test rigorously whether they all lie in the disc $|z| < \sqrt{2}$. If so, we may then rigorously compute $\qual(Q(z)) \in \overline{\QQ}$ as follows.
Since $Q(z) Q(\overline{z})$ is a symmetric integer polynomial in $z$ and $\overline{z}$, we may rewrite it as an integer polynomial in $z+\overline{z}$ and $z \overline{z}$; specializing these to $t$ and $2$, respectively, yields an integer polynomial $R(t)$ such that
\begin{equation} \label{eq:quality formula}
\qual(Q(z))^2 = \min\{R(t): t \in [-2\sqrt{2}, 2\sqrt{2}]\}.
\end{equation}
The minimum is achieved either at $\pm 2 \sqrt{2}$ or at some zero of $R'(t)$ in $[-2 \sqrt{2}, 2 \sqrt{2}]$.

We now describe the main computation. We first run an exhaust over monic polynomials of degree at most 7 with all coefficients in $\{-3, \dots, 3\}$, checking whether each polynomial is compliant.
(While it would be feasible to perform an exhaustive search for compliant polynomials of degree up to 7 by adapting the search strategy for Weil polynomials described in \cite{kedlaya-search}, we did not need to implement this here.)
In this way we find compliant representations of 167 distinct integers in the range $\{1,\dots,459\}$;
for each of these we record the maximum observed quality, rounded down to the nearest multiple of $1/7$.

We then compute, for each odd integer $m \in \{1,\dots,50000\}$, a lower bound on the maximum quality of a compliant representation of $m$ using the following logic.
Given odd integers $m_1 < m$, let $m_2$ be one of the nearest odd integers to $m/m_1$ and set $c = m - m_1 m_2$. Let $Q_1(z), Q_2(z)$ be compliant representations of $m_1, m_2$ of respective qualities $q_1, q_2$.
If $q_1 q_2 > |c|$, then $R(z) = Q_1(z) Q_2(z) + c$ is a compliant representation of $m$ of quality at least $q_1 q_2 - |c|$. 

From the results of this computation, we read off (a) and (b).
\end{proof}

\begin{lemma} \label{lem:polynomials with roots in interval odd case}
Every positive odd integer admits a compliant representation. Hence by Lemma~\ref{lem:compliant representation}
(or the theorem of Madan--Pal in the case $m=1$), Theorem~\ref{thm:AV by order} holds for $m$ odd.
\end{lemma}
\begin{proof}
By Lemma~\ref{lem:compliant representation}(a), it will suffices to check that every odd integer $m \geq 3095$ admits a compliant representation of quality at least 7. We check by induction on $m$ that each of $\{m, m+2, \dots, 15m-16\}$ admits such a representation, this being true for $m = 3095$ by Lemma~\ref{lem:compliant representation}(b) because $15 \cdot 3095 - 16 < 50000$. 

Given the claim for some $m$, let $Q(z)$ be a compliant representation of $m$ of quality at least 7. 
For $c$ even with $|c| \leq 14$,
\[
R(z) = (z^4 - 1)Q(z) + c
\]
is a compliant representation of $15m+c$ of quality at least
\[
\qual(z^4-1)\qual(Q(z)) - |c| \geq 3 \cdot 7 - 14 \geq 7.
\]
It follows that each of $\{m+2, m+4, \dots, 15m + 14 = 15(m+2) - 16\}$ admits a compliant representation of quality at least $7$;
that is, the induction hypothesis holds with $m$ replaced by $m+2$, as desired.
\end{proof}

\vfill

\appendix
\pagebreak
\section{\SageMath{} code for Lemma~\ref{lem:compliant representation}}

This code uses the following features of Python and \SageMath{}:
\begin{itemize}
\item
\verb+all+ returns \verb+True+ iff all of its inputs evaluate to \verb+True+ in a boolean context.
\item
Python indexing starts from 0 rather than 1, so \verb+range(n)+ returns $0,\dots,n-1$ and \verb+range(1,n)+ returns $1,\dots,n-1$.
\item
\verb+AA+ and \verb+QQbar+ are predefined in \SageMath{} as the fields of algebraic real and complex numbers, respectively. Computations in these fields is rigorous, not subject to roundoff errors.
\item
For \verb+f+ a polynomial over a field, \verb+f.roots(K)+ computes its roots in the field \verb+K+ (defaulting to the base field of \verb+f+ if \verb+K+ is omitted). The output consists of pairs $(\alpha, m)$ where $\alpha$ is a root and $m$ is the multiplicity of the root.
\end{itemize}

{
\begin{verbatim}
import itertools
R.<z> = QQ[] # Univariate polynomial ring

# Check that the polynomial f has all complex roots in the disc |z| < sqrt(2). 
def all_roots_in_disc(f):
    return all(abs(i)^2 < 2 for (i,_) in f.roots(QQbar))

# Compute the quality of f (see (8.4)), multiplied by 7 
# and rounded down to the nearest integer.
def quality_lower_bound(f):
    P.<x,y,t> = QQ[]
    I = P.ideal(x+y-t, x*y-2)
    # Compute a representative of f(x)*f(y) modulo I.
    # This corresponds to the polynomial R(t) appearing in the proof of Lemma 8.5.
    g1 = I.reduce(f(x) * f(y))
    # The polynomial g1 is currently a univariate polynomial in t, but in the ring P.
    # We next create g by substituting t -> z to land in the ring R.
    g = R(g1(0, 0, z))
    # Make the list of roots of this polynomial, together with +/- 2*sqrt(2).
    rootlist = (g.derivative()*(z^2 - 8)).roots(AA)
    # Implement (8.7).
    ans = min((g(i) * 49).floor() for (i,_) in rootlist if i^2 <= 8) 
    return floor(sqrt(ans))

# Create a table of compliant representations of small integers.
compliant_reps = {}
rep_quality = {}
for n in range(1, 8): # step through n = 1, ..., 7
    # Iterate over n-tuples in which the i-th term (starting wih i=0) runs over 
    # -3,-1,1,3 if (n choose i) is odd and -2,0,2 otherwise.
    for t in itertools.product(*((range(-3,4,2) if binomial(n,i)%2 
                                  else range(-2,3,2)) for i in range(n))):
        # Convert t to a polynomial, after appending 1 for the leading coefficient.
        u = R(list(t) + [1]) 
        if all_roots_in_disc(u):
            m = u(2)
            q = quality_lower_bound(u)
            if m not in rep_quality or rep_quality[m] < q:
                compliant_reps[m] = u
                rep_quality[m] = q
\end{verbatim}
\pagebreak

\begin{verbatim}
# Compute lower bounds of qualities of compliant representations of larger integers.
# To save time, rather than optimizing fully, we quit as soon as we find a representation
# of quality at least 8.
n = 50000
for m in range(1, n, 2): # step by 2
    for m1 in range(3, ceil(sqrt(m)), 2): # step by 2
        if m1 in rep_quality:
            # Let m2 be one of the nearest odd integers to m/m1.
            tmp = (QQ(m)/m1+1) / 2
            for m2 in [tmp.floor()*2-1, tmp.ceil()*2-1]:
                if m2 < m and m2 in rep_quality:
                    c = m - m1*m2
                    q = (rep_quality[m1]*rep_quality[m2])//7 - abs(c)
                    if m not in rep_quality or rep_quality[m] < q:
                        compliant_reps[m] = compliant_reps[m1]*compliant_reps[m2] + c
                        if q < 56:
                            q = quality_lower_bound(compliant_reps[m])
                        rep_quality[m] = q
            if m in rep_quality and rep_quality[m] >= 56:
                break

# Running these commands without errors confirms the conclusions of Lemma 8.5.
assert all(i in rep_quality for i in range(1, n, 2))
assert all(rep_quality[i] >= 49 for i in range(3095, n, 2))
\end{verbatim}
}

\pagebreak

\end{document}